\documentclass[final,3p,12pt]{elsarticle}

\usepackage{amssymb}
\usepackage{amsmath}
\usepackage{amsthm}

\journal{Linear Algebra and its Applications}

\newcommand{\bbC}{\mathbb{C}}
\newcommand{\bbH}{\mathbb{H}}
\newcommand{\bbR}{\mathbb{R}}
\newcommand{\bbZ}{\mathbb{Z}}
\newcommand{\calB}{\mathcal{B}}
\newcommand{\calK}{\mathcal{K}}
\newcommand{\rme}{\mathrm{e}}
\newcommand{\rmi}{\mathrm{i}}
\newcommand{\rmI}{\mathrm{I}}
\newcommand{\rmO}{\mathrm{O}}

\newcommand{\rmT}{\mathrm{T}}
\newcommand{\Tr}{\mathrm{Tr}}

\newcommand{\abs}[1]{|{#1}|}

\newcommand{\set}[1]{\{{#1}\}}

\newcommand{\norm}[1]{\|{#1}\|}

\newcommand{\ip}[2]{\langle{#1},{#2}\rangle}

\newtheorem{thm}{Theorem}

\newproof{pf}{Proof}
\theoremstyle{definition}

\begin{document}
\begin{frontmatter}
\title{Steiner equiangular tight frames}

\author[AFIT]{Matthew Fickus}
\ead{Matthew.Fickus@afit.edu}
\author[Princeton]{Dustin G.~Mixon}
\author[MO]{Janet C.~Tremain}

\address[AFIT]{Department of Mathematics and Statistics, Air Force Institute of Technology\\Wright-Patterson Air Force Base, Ohio 45433, USA}
\address[Princeton]{Program in Applied and Computational Mathematics, Princeton University\\Princeton, New Jersey 08544, USA}
\address[MO]{Department of Mathematics, University of Missouri, Columbia, Missouri 65211, USA}

\begin{abstract}
We provide a new method for constructing equiangular tight frames (ETFs).  The construction is valid in both the real and complex settings, and shows that many of the few previously-known examples of ETFs are but the first representatives of infinite families of such frames.  It provides great freedom in terms of the frame's size and redundancy.  This method also explicitly constructs the frame vectors in their native domain, as opposed to implicitly defining them via their Gram matrix. Moreover, in this domain, the frame vectors are very sparse.  The construction is extremely simple: a tensor-like combination of a Steiner system and a regular simplex.  This simplicity permits us to resolve an open question regarding ETFs and the restricted isometry property (RIP): we show that the RIP behavior of some ETFs is unfortunately no better than their coherence indicates. 
\end{abstract}

\end{frontmatter}

%%%%%%%%%%%%%%%%%%%%%%%%%%%%%%%%%%%%%%%%%%%%%%%%%%%%%%%%%%%%%%%%
%%%%%%%%%%%%%%%%%%%%%%%%%%%%%%%%%%%%%%%%%%%%%%%%%%%%%%%%%%%%%%%%
% Section 1: Introduction
%%%%%%%%%%%%%%%%%%%%%%%%%%%%%%%%%%%%%%%%%%%%%%%%%%%%%%%%%%%%%%%%
%%%%%%%%%%%%%%%%%%%%%%%%%%%%%%%%%%%%%%%%%%%%%%%%%%%%%%%%%%%%%%%%

\section{Introduction}
Let $F=\set{f_n}_{n=1}^{N}$ be a finite sequence of vectors in a real or complex $M$-dimensional Hilbert space $\bbH_M$.  The corresponding \textit{frame operator} is $FF^*=\sum_{n=1}^{N}f_nf_n^*$, where $f_n^*$ denotes the linear functional that maps a given $f\in\bbH_M$ to the scalar $\ip{f}{f_n}$.  The sequence $F$ is said to be a \textit{tight frame} if there exists $A>0$ such that $FF^*=A\rmI$.  Meanwhile, $F$ is \textit{equiangular} if $\norm{f_n}=1$ for all $n$ and if there exists $\alpha\geq0$ such that $\abs{\ip{f_n}{f_{n'}}}=\alpha$ for all $n\neq n'$.  This paper concerns \textit{equiangular tight frames} (ETFs); writing $F$ as an $M\times N$ matrix, we need the rows of $F$ to be orthogonal and have constant norm, the columns of $F$ to be unit norm, and the inner products of distinct columns of $F$ to have constant modulus.  As detailed below, such frames are useful in applications, but up to this point, they have proven notoriously difficult to construct. 

In this article, we provide a new method for constructing ETFs.  The construction is valid in both the real and complex settings, and shows that many of the few previously-known examples of ETFs are but the first representatives of infinite families of such frames.  This construction technique also permits great freedom in selecting $M$ and $N$, just shy of letting one choose the exact size and redundancy of their liking.  This method also explicitly constructs the frame vectors in their native domain $\bbH_M$, as opposed to the usual method of implicitly defining them with their \textit{Gram matrix} $F^*F$.  Moreover, in this domain, the frame vectors can be chosen to be very sparse.  The construction is extremely simple: a tensor-like combination of a Steiner system and a regular simplex.  This simplicity permits us to resolve an open question regarding ETFs and the restricted isometry property (RIP): we show that the RIP behavior of some ETFs is unfortunately no better than their worst-case-coherence bounds indicate. 

Equiangular lines have long been a subject of interest~\cite{LemmensS73}.  Recent work on the matter of ETFs was spurred on by communications-theory-inspired results~\cite{BodmannP05,HolmesP04,StrohmerH03} that show that the linear encoders provided by such frames are optimally robust against channel erasures.  In the real setting, the existence of an ETF of a given size is equivalent to the existence of a strongly regular graph with certain corresponding parameters~\cite{HolmesP04,Seidel73}.  Such graphs have a rich history and remain an active topic of research~\cite{Brouwer07}; the specific ETFs that arise from particular graphs are detailed in~\cite{Waldron09}.  Some of this theory generalizes to the complex-variable setting in the guise of complex Seidel matrices~\cite{BodmannE10,BodmannPT09,DuncanHS10}.  Many approaches to constructing ETFs have focused on the special case in which every entry of $F$ is a root of unity~\cite{Kalra06,Renes07,Strohmer08,SustikTDH07,XiaZG05}.  Other approaches are given in~\cite{CasazzaRT08,Singh10,TroppDHS05}.  In the complex setting, much attention has focused on the \textit{maximal} case of $M^2$ vectors in $\bbH_M$~\cite{Appleby05,Fickus09,Khatirinejad08,RenesBSC04,ScottG10}.

A version of the ETF construction method we present here was previously employed by Seidel in Theorem~12.1 of~\cite{Seidel73} to prove the existence of certain strongly regular graphs.  In the context of that result, our contributions are: (i) the realization that when Seidel's block design arises from a particular type of Steiner system, the resulting strongly regular graph indeed corresponds to a real ETF; (ii) noting that in this case, the graph theory may be completely bypassed, as the idea itself directly produces the requisite frame $F$; and (iii) having bypassed the graph theory, realizing that this construction immediately generalizes to the complex-variable setting if Seidel's requisite Hadamard matrix is permitted to become complex.  These realizations permit us to exploit the vast literature on Steiner systems~\cite{ColbournM07} to construct several new infinite families of ETFs, in both the real and complex settings.  Moreover, these ETFs are extremely sparse in their native space; sparse tight frames have recently become a subject of interest in their own right~\cite{CasazzaHKK10}.

In fact, these ETFs are simple enough so as to permit a rigorous investigation of their potential as RIP matrices, which are currently in demand due to their applicability in compressed sensing~\cite{CandesT05,CandesT06}.  As discussed below, ETFs are the optimal matrices with respect to a very coarse estimate---worst-case coherence---on a matrix's RIP bounds.  Our hope was that all ETFs, having such high degrees of symmetry, might possess other hidden properties that, when properly exploited, yield even better bounds than those given by coherence-based estimates.  Unfortunately, our newly-discovered ETF constructions dash these hopes: for at least some ETFs, worst-case coherence, \`{a} la Gershgorin circles, does indeed provide a very good estimate on RIP bounds.  With respect to RIP, these ETFs perform no better than a myriad of previously discovered deterministic constructions of RIP matrices, such as those given in~\cite{DeVore07}.

In the next section, we provide our main result, namely Theorem~\ref{theorem.steiner etfs}, which shows how certain Steiner systems may be combined with regular simplices to produce ETFs.  In the third section, we discuss each of the known infinite families of such Steiner systems, and compute the corresponding infinite families of ETFs they generate.  We further provide some necessary and asymptotically sufficient conditions, namely Theorem~\ref{theorem.necessary conditions}, to aid in the quest for discovering other examples of such frames that lie outside of the known infinite families.   In Section~$4$, we discuss the possible RIP behavior of ETFs in general, and show that the performance of our Steiner ETFs is indeed no better than that guaranteed by coherence-based estimates.

\section{Steiner equiangular tight frames}

In this section, we provide new constructions of infinite families of ETFs, namely $M\times N$ matrices $F=[f_1 \dotsc f_N]$ which have orthogonal rows of constant squared-norm $A$ and unit norm columns whose inner products have constant modulus $\alpha$: we want $FF^*=A\rmI$ while the diagonal entries of $F^*F$ are $1$ and the off-diagonal entries are $\alpha$ in modulus.  For a fixed $M$ and $N$, there is no ambiguity~\cite{StrohmerH03} as to the values of $A$ and $\alpha$.  Indeed, noting that since $N=\sum_{n=1}^{N}\norm{f_n}^2=\Tr(F^*F)=\Tr(FF^*)=MA$, we have $A=\frac NM$; moreover, since
\begin{equation*}
N+N(N-1)\alpha^2=\sum_{n,n'=1}^{N}\abs{\ip{f_n}{f_{n'}}}^2=\Tr[(F^*F)^2]=\Tr[(FF^*)^2]=MA^2=\tfrac{N^2}{M},
\end{equation*}
then $\smash{\alpha^2=\frac{N-M}{M(N-1)}}$.  Conversely, if one can design an $N\times N$ self-adjoint, positive semidefinite Gram matrix $G$ of rank $M$ whose diagonal entries are one and whose off-diagonal entries are all $\smash{\frac{N-M}{M(N-1)}}$ in squared-modulus, then one can then factor $G$ as $F^*F$, where $F$ is an $M\times N$ ETF~\cite{StrohmerH03}.  This fact has led many to attempt to construct ETFs, not by constructing $F$ directly, but rather, by constructing $G$.  This Gram representation of an ETF has the additional benefit of being invariant with respect to rotations of the frame elements themselves, and, in the real-variable case, is closely-related to the incidence matrix of the corresponding strongly regular graph~\cite{HolmesP04}.  Moreover, whenever $G=F^*F$ and $FF^*=A\rmI$, then the columns of $G$ are, in fact, a scalar multiple of an isometric embedding of the columns of $F$.  That is, the columns of the Gram matrix of a tight frame are but a high-dimensional representation of the frame elements themselves.

There is a drawback, however, to working with Gram representations: one does not produce the frame vectors in their native $M$-dimensional space, the domain in which they are usually needed for communications applications.  And though factoring $G$ is straightforward---one may, for instance, apply the Gram-Schmidt algorithm to the columns of $G$---it will produce the $f_n$'s with respect to some arbitrarily chosen basis for $\bbH_M$, one that may not be optimal.  Indeed, such a process ignores the question of whether or not there is a basis for $\bbH_M$ that makes the frame elements sparse.  For this reason, in this paper, we avoid the Gram representation and construct $F$ directly.  The key idea is to design the ETFs in blocks, specifically those arising from a particular type of combinatorial block design.

Steiner systems and block designs have been studied for over a century; the background facts presented here on these topics are taken from~\cite{AbelG07,ColbournM07}.  In short, a $(v,b,r,k,\lambda)$-\textit{block design} is a $v$-element set $V$ along with a collection $\calB$ of $b$ $k$-element subsets of $V$, dubbed \textit{blocks}, that have the property that any element of $V$ lies in exactly $r$ blocks and that any $2$-element subset of $V$ is contained in exactly $\lambda$ blocks.  The corresponding \textit{incidence matrix} is a $v\times b$ matrix $A$ that has a one in a given entry if that block contains the corresponding point, and is otherwise zero; in this paper, it is more convenient for us to work with the $b\times v$ transpose $A^\rmT$ of this incidence matrix.  Our particular construction of ETFs involves a special class of block designs known as $(2,k,v)$-\textit{Steiner systems}.  These have the property that any $2$-element subset of $V$ is contained in exactly one block, that is, $\lambda=1$.  With respect to our purposes, the crucial facts are the following:
\begin{quote}
The transpose $A^\rmT$ of the $\set{0,1}$-incidence matrix $A$ of a $(2,k,v)$-Steiner system:
\begin{enumerate}
\item[(i)]
is of size $\smash{\frac{v(v-1)}{k(k-1)}}\times v$,
\item[(ii)]
has $k$ ones in each row,
\item[(iii)]
has $\smash{\frac{v-1}{k-1}}$ ones in each column, and
\item[(iv)]
has the property that any two of its columns have a dot product of one.
\end{enumerate} 
\end{quote}
The first three facts follow immediately from solving for $\smash{b=\frac{v(v-1)}{k(k-1)}}$ and $\smash{r=\frac{v-1}{k-1}}$, using the well-known relations $vr=bk$ and $r(k-1)=\lambda(v-1)$.  Meanwhile, (iv) comes from the fact that $\lambda=1$: each column of $A^\rmT$ corresponds to an element of the set, and the dot product of any two columns computes the number of blocks that contains the corresponding pair of points.  This in hand, we present our main result;  here, the \textit{density} of a matrix is the ratio of the number of nonzero entries of that matrix to the entire number of its entries:
\begin{thm}
\label{theorem.steiner etfs}
Every $(2,k,v)$-Steiner system generates an equiangular tight frame consisting of $N=v(1+\frac{v-1}{k-1})$ vectors in $M=\frac{v(v-1)}{k(k-1)}$-dimensional space with redundancy $\smash{\frac NM=k(1+\frac{k-1}{v-1})}$ and density $\smash{\frac{k}{v}=(\frac{N-1}{M(N-M)})^{\frac12}}$.\medskip

\noindent
Moreover, if there exists a real Hadamard matrix of size $\smash{1+\frac{v-1}{k-1}}$, then such frames are real.\medskip

\noindent
Specifically, a $\frac{v(v-1)}{k(k-1)}\times v(1+\frac{v-1}{k-1})$ ETF matrix $F$ may be constructed as follows:
\begin{enumerate}
\item
Let $A^\rmT$ be the $\smash{\frac{v(v-1)}{k(k-1)}}\times v$ transpose of the adjacency matrix of a $(2,k,v)$-Steiner system.\smallskip
\item
For each $j=1,\dotsc,v$, let $H_j$ be any $(1+\frac{v-1}{k-1})\times(1+\frac{v-1}{k-1})$ matrix that has orthogonal rows and unimodular entries, such as a possibly complex Hadamard matrix.\smallskip
\item
For each $j=1,\dotsc,v$, let $F_j$ be the $\smash{\frac{v(v-1)}{k(k-1)}\times(1+\frac{v-1}{k-1})}$ matrix obtained from the $j$th column of $A^\rmT$ by replacing each of the one-valued entries with a distinct row of $H_j$, and every zero-valued entry with a row of zeros.\smallskip
\item
Concatenate and rescale the $F_j$'s to form  $F=(\frac{k-1}{v-1})^\frac12[F_1 \cdots F_v]$.
\end{enumerate}
\end{thm}
We refer to the ETFs produced by Theorem~\ref{theorem.steiner etfs} as \textit{$(2,k,v)$-Steiner ETFs}.  In essence, the idea of the construction is to realize that the nonzero rows of any particular $F_j$ form a regular simplex in $\smash{\frac{v-1}{k-1}}$-dimensional space; these vectors are automatically equiangular amongst themselves; by requiring the entries of these simplices to be unimodular, and requiring that distinct blocks have only one entry of mutual support, one can further control the inner products of vectors arising from distinct blocks.  This idea is best understood by considering a simple example, such as the ETF that arises from a $(2,2,4)$-Steiner system whose transposed incidence matrix is:
\begin{equation*}
A^\rmT=\begin{bmatrix}+&+&&\\+&&+&\\+&&&+\\&+&+&\\&+&&+\\&&+&+\end{bmatrix}.
\end{equation*}
One can immediately verify that $A^\rmT$ corresponds to a block design: there is a set $V$ of $v=4$ elements, each corresponding to a column of $A^\rmT$; there is also a collection $\calB$ of $b=6$ subsets of $V$, each corresponding to a row of $A^\rmT$; every row contains $k=2$ elements; every column contains $r=3$ elements; any given pair of elements is contained in exactly one row, that is, $\lambda=1$, a fact which is equivalent to having the dot product of any two distinct columns of $A^\rmT$ being one.  To form an ETF, for each of the four columns of $A^\rmT$ we must choose a $4\times 4$ matrix $H$ with unimodular entries and orthogonal rows; the size of $H$ is always one more than the number $r$ of ones in a given column of $A^\rmT$.  Though in principle one may choose a different $H$ for each column, we choose them all to be the same, namely the Hadamard matrix:
\begin{equation*}
H=\begin{bmatrix}+&+&+&+\\+&-&+&-\\+&+&-&-\\+&-&-&+\end{bmatrix}.
\end{equation*}
To form the ETF, for each column of $A^\rmT$ we replace each of its $1$-valued entries with a distinct row of $H$.  Again, though in principle one may choose a different sequence of rows of $H$ for each column, we simply decide to use the second, third and fourth rows, in that order.  The result is a real ETF of $N=16$ elements of dimension $M=6$:
\begin{equation*}
F=\frac1{\sqrt{3}}\left[\begin{array}{cccccccccccccccc}+&-&+&-&+&-&+&-\\+&+&-&-&&&&&+&-&+&-\\+&-&-&+&&&&&&&&&+&-&+&-\\&&&&+&+&-&-&+&+&-&-\\&&&&+&-&-&+&&&&&+&+&-&-\\&&&&&&&&+&-&-&+&+&-&-&+  \end{array}\right].
\end{equation*}
One can immediately verify that the rows of $F$ are orthogonal and have constant norm, implying $F$ is indeed a tight frame.  One can also easily see that the inner products of two columns from the same block are $-\frac13$, while the inner products of columns from distinct blocks are $\pm\frac13$.  Theorem~\ref{theorem.steiner etfs} states that this behavior holds in general for any appropriate choice of $A^\rmT$ and $H$; its formal proof is as follows.
\begin{proof}
To verify $F$ is a tight frame, note that the inner product of any two distinct rows of $F$ is zero, as they are the sum of the inner products of the corresponding rows of the $F_j$'s over all $j=1,\dotsc,v$; for any $j$, these shorter inner products are necessarily zero, as they either correspond to inner products of distinct rows of $H_j$ or to inner products with zero vectors.  Moreover, the rows of $F$ have constant norm: as noted in (ii) above, each row of $A^\rmT$ contains $k$ ones; since each $H_j$ has unimodular entries, the squared-norm of any row of $F$ is the squared-scaling factor $\frac{k-1}{v-1}$ times a sum of $\smash{k(1+\frac{v-1}{k-1})}$ ones, which, as is necessary for any unit norm tight frame, equals the redundancy $\smash{\frac NM=k(1+\frac{k-1}{v-1})}$.

Having that $F$ is tight, we show $F$ is also equiangular.  We first note that the columns of $F$ have unit norm: the squared-norm of any column of $F$ is $\smash{\frac{k-1}{v-1}}$ times the squared-norm of a column of one of the $F_j$'s; since the entries of $H_j$ are unimodular and (iii) above gives that each column of $A^\rmT$ contains $\smash{\frac{v-1}{k-1}}$ ones, the squared-norm of any column of $F$ is $\smash{(\frac{k-1}{v-1})(\frac{v-1}{k-1})1=1}$, as claimed.  Moreover, the inner products of any two distinct columns of $F$ has constant modulus.  Indeed, the fact (iv) that any two distinct columns of $A^\rmT$ have but a single entry of mutual support implies the same is true for columns of $F$ that arise from distinct $F_j$ blocks, implying the inner product of such columns is $\smash{\frac{k-1}{v-1}}$ times the product of two unimodular numbers.  That is, the squared-magnitude of the inner products of two columns that arise from distinct blocks is $\smash{\frac{N-M}{M(N-1)}=(\frac{k-1}{v-1})^2}$, as needed.  Meanwhile, the same holds true for columns that arise from the same block $F_j$.  To see this, note that since $H_j$ is a scalar multiple of a unitary matrix, its columns are orthogonal.  Moreover, $F_j$ contains all but one of the $H_j$'s rows, namely one for each of the $1$-valued entries of $A^\rmT$, \`{a} la (iii).  Thus, the inner products of the portions of $H_j$ that lie in $F_j$ are their entire inner product of zero, less the contribution from the left-over entries.  Overall, the inner product of two columns of $F$ that arise from the same $F_j$ block is $\smash{\frac{k-1}{v-1}}$ times the negated product of one entry of $H_j$ and the conjugate of another; since $H_j$ is unimodular, we have that the squared-magnitude of such inner products is $\smash{\frac{N-M}{M(N-1)}=(\frac{k-1}{v-1})^2}$, as needed.

Thus $F$ is an ETF.  Moreover, as noted above, its redundancy is $\smash{\frac NM=k(1+\frac{k-1}{v-1})}$.  All that remains to verify is its density: as the entries of each $H_j$ are all nonzero, the proportion of $F$'s nonzero entries is the same as that of the incidence matrix $A$, which is clearly $\frac kv$, having $k$ ones in each $v$-dimensional row.  Moreover, substituting $\smash{N=v(1+\frac{v-1}{k-1})}$ and $\smash{M=\frac{v(v-1)}{k(k-1)}}$ into the quantity $\smash{\frac{N-1}{M(N-M)}}$ reveals it to be $\frac{k^2}{v^2}$, and so the density can be alternatively expressed as $\smash{(\frac{N-1}{M(N-M)})^{\frac12}}$, as claimed.
\end{proof}

In the next section, we apply Theorem~\ref{theorem.steiner etfs} to produce several infinite families of Steiner ETFs.  Before doing so, however, we pause to remark on the redundancy and sparsity of such frames.  In particular, note that since the parameters $k$ and $v$ of the requisite Steiner system always satisfy $2\leq k\leq v$, then the redundancy $k(1+\frac{k-1}{v-1})$ of Steiner ETFs is always between $k$ and $2k$; the redundancy is therefore on the order of $k$, and is always strictly greater than $2$.  If a low-redundancy ETF is desired, one can always take the Naimark complement~\cite{CasazzaFMWZ10} of an ETF of $N$ elements in $M$-dimensional space to produce a new ETF of $N$ elements in $(N-M)$-dimensional space; though the complement process does not preserve sparsity, it nevertheless transforms any Steiner ETF into a new ETF whose redundancy is strictly less than $2$.  However, such a loss of sparsity should not to be taken lightly.  Indeed, the low density of Steiner ETFs gives them a large computational advantage over their non-sparse brethren.  

To clarify, the most common operation in frame-theoretic applications is the evaluation of the \textit{analysis} operator $F^*$ on a given $f\in\bbH_M$.  For a non-sparse $F$, this act of computing $F^*f$ requires $\rmO(MN)$ operations; for a frame $F$ of density $D$, this cost is reduced to $\rmO(DMN)$.  Indeed, using the explicit value of $\smash{D=(\frac{N-1}{M(N-M)})^{\frac12}}$ given in Theorem~\ref{theorem.steiner etfs} as well as the aforementioned fact that the redundancy of such frames necessarily satisfies $\frac NM>2$, we see that the cost of evaluating $F^*f$ when $F$ is a Steiner ETF is on the order of $\smash{(\frac{M(N-1)}{N-M})^{\frac12}N <(2M)^\frac{1}{2}N}$ operations, a dramatic cost savings when $M$ is large.  Further efficiency is gained when $F$ is real, as its nonzero elements are but a fixed scaling factor times the entries of a real Hadamard matrix, implying $F^*f$ can be evaluated using only additions and subtractions.  The fact that every entry of $F$ is either $0$ or $\pm 1$ further makes real Steiner ETFs potentially useful for applications that require binary measurements, such as design of experiments.

\section{Examples of Steiner equiangular tight frames}
\label{section.examples of Steiner ETFs}
In this section, we apply Theorem~\ref{theorem.steiner etfs} to produce several infinite families of Steiner ETFs.   When designing frames for real-world applications, three considerations reign supreme: size, redundancy and sparsity.  As noted above, every Steiner ETF is very sparse, a serious computational advantage in high-dimensional signal processing.  Moreover, some of these infinite families, such as those arising from finite affine and projective geometries, provide one great flexibility in choosing the ETF's size and redundancy.  Indeed, these constructions provide the first known guarantee that for a given application, one is always able to find ETFs whose frame elements lie in a space whose dimension matches, up to an order of magnitude, that of one's desired class of signals, while simultaneously permitting one to have an almost arbitrary fixed level of redundancy, a handy weapon in the fight against noise.  To be clear, recall that the redundancy of a Steiner ETF is always strictly greater than $2$.  Moreover, as general bounds on the maximal number of equiangular lines~\cite{LemmensS73} require that any ETF satisfy $\smash{N\leq\frac{M(M+1)}{2}}$ in real spaces and $N\leq M^2$ in complex ones, the redundancy of an ETF is never truly arbitrary.  Nevertheless, if one does prescribe a given desired level of redundancy in advance, the Steiner method can produce arbitrarily large ETFs whose redundancy is approximately the prime power nearest to the sought-after level.  

\subsection{Infinite families of Steiner equiangular tight frames}
We now detail eight infinite families of ETFs, each generated by applying Theorem~\ref{theorem.steiner etfs} to one of the eight completely understood infinite families of $(2,k,v)$-Steiner systems.  Table~\ref{table.infinite families} summarizes the most important features of each family, while Table~\ref{table.low-dimensional examples} gives the first few examples of each type, summarizing those that lie in 100 dimensions or less.

\subsubsection{All two-element blocks: $(2,2,v)$-Steiner ETFs for any $v\geq2$.}

The first infinite family of Steiner systems is so simple that it is usually not discussed in the design-theory literature.  For any $v\geq 2$, let $V$ be a $v$-element set, and let $\calB$ be the collection of all $2$-element subsets of $V$.  Clearly, we have $\smash{b=\frac{v(v-1)}{2}}$ blocks, each of which contains $k=2$ elements; each point is contained in $r=v-1$ blocks, and each pair of points is indeed contained in but a single block, that is, $\lambda=1$.

By Theorem~\ref{theorem.steiner etfs}, the ETFs arising from these $(2,2,v)$-Steiner systems consist of $N=v(1+\frac{v-1}{k-1})=v^2$ vectors in $\smash{M=\frac{v(v-1)}{k(k-1)}=\frac{v(v-1)}{2}}$-dimensional space.  Though these frames can become arbitrarily large, they do not provide any freedom with respect to redundancy: $\smash{\frac NM=2\frac{v}{v-1}}$ is essentially $2$.  These frames have density $\smash{\frac kv=\frac2v}$.  Moreover, these ETFs can be real-valued if there exists a real Hadamard matrix of size $\smash{1+\frac{v-1}{k-1}}=v$.  In particular, it suffices to have $v$ to be a power of $2$; should the Hadamard conjecture prove true, it would suffice to have $v$ divisible by $4$.

One example of such an ETF with $v=4$ was given in the previous section.  For another, consider $v=3$.  The $b\times v$ transposed incidence matrix $A^\rmT$ is $3\times 3$, with each row corresponding to a given $2$-element subset of $\set{0,1,2}$:
\begin{equation*}
A^\rmT=\begin{bmatrix}+&+&\\+&&+\\&+&+\end{bmatrix}.
\end{equation*}
To form the corresponding $3\times 9$ ETF $F$, we need a $3\times 3$ unimodular matrix with orthogonal rows, such as a DFT; letting $\smash{\omega=\rme^{2\pi\rmi/3}}$, we can take
\begin{equation*}
H=\left[\begin{array}{lll}1&1&1\\1&\omega^2&\omega\\1&\omega&\omega^2\end{array}\right].
\end{equation*}
To form $F$, in each column of $A^\rmT$, we replace each $1$-valued entry with a distinct row of $H$.  Always choosing the second and third rows yields an ETF of $9$ elements in $\bbC^3$:
\begin{equation*}
F=\frac1{\sqrt{2}}\left[\begin{array}{lllllllll}1&\omega^2&\omega&1&\omega^2&\omega&&&\\1&\omega&\omega^2&&&&1&\omega^2&\omega\\&&&1&\omega&\omega^2&1&\omega&\omega^2\end{array}\right].
\end{equation*}
This is the only known instance of when the Steiner-based construction of Theorem~\ref{theorem.steiner etfs} produces a maximal ETF, namely one that has $N=M^2$.

\subsubsection{Steiner triple systems: $(2,3,v)$-Steiner ETFs for any $v\equiv 1,3\mod 6$.}

\textit{Steiner triple systems}, namely $(2,3,v)$-Steiner systems, have been a subject of interest for over a century, and are known to exist precisely when $v\equiv 1,3\mod 6$~\cite{ColbournM07}.  Each of the $\smash{b=\frac{v(v-1)}{6}}$ blocks contains $k=3$ points, while each point is contained in $\smash{r=\frac{v-1}{2}}$ blocks.  The corresponding ETFs produced by Theorem~\ref{theorem.steiner etfs} consist of $\smash{\frac{v(v+1)}{2}}$ vectors in $\smash{\frac{v(v-1)}6}$-dimensional space.  The density of such frames is $\frac3v$.  As with ETFs stemming from $2$-element blocks, Steiner triple systems offer little freedom in terms of redundancy: $\smash{\frac NM=3\frac{v+1}{v-1}}$ is always approximately $3$.  Such ETFs can be real if there exists a real Hadamard matrix of size $\smash{\frac{v+1}{2}}$.

The \textit{Fano plane} is a famous example of such a design.  The simplest example of a finite projective geometry, it consists of $v=7$ points and $b=7$ lines, any two of which intersect in exactly one point.  Each line consists of $k=3$ points, and each point is contained in $r=3$ lines:
\begin{equation*}
A^\rmT=\begin{bmatrix}+&+&+\\+&&&+&+\\+&&&&&+&+\\&+&&+&&+\\&+&&&+&&+\\&&+&+&&&+\\&&+&&+&+\end{bmatrix}.
\end{equation*}
Choosing $H$ to be the standard $4\times 4$ Hadamard matrix used in the previous section results in a real ETF of $28$ elements in $7$-dimensional space; $F$ is the scaling factor $\smash{\frac1{\sqrt{3}}}$ times:
\begin{equation*}
\begin{tiny}
\left[\begin{array}{cccccccccccccccccccccccccccc}
+&-&+&-&+&-&+&-&+&-&+&-& & & & & & & & & & & & & & & & \\
+&+&-&-& & & & & & & & &+&-&+&-&+&-&+&-& & & & & & & & \\
+&-&-&+& & & & & & & & & & & & & & & & &+&-&+&-&+&-&+&-\\
 & & & &+&+&-&-& & & & &+&+&-&-& & & & &+&+&-&-& & & & \\
 & & & &+&-&-&+& & & & & & & & &+&+&-&-& & & & &+&+&-&-\\
 & & & & & & & &+&+&-&-&+&-&-&+& & & & & & & & &+&-&-&+\\
 & & & & & & & &+&-&-&+& & & & &+&-&-&+&+&-&-&+& & & & 
\end{array}\right]
\end{tiny}.
\end{equation*}

\subsubsection{Four element blocks: $(2,4,v)$-Steiner ETFs for any $v\equiv 1,4\mod 12$.}

It is known that $(2,4,v)$-Steiner systems exist precisely when $v\equiv 1,4\mod 12$~\cite{AbelG07}.  Continuing the trend of the previous two families, these ETFs can vary in size but not in redundancy: they consist of $\smash{\frac{v(v+2)}{3}}$ vectors in $\smash{\frac{v(v-1)}{12}}$-dimensional space, having redundancy $\smash{4\frac{v+2}{v-1}}$ and density $\smash{\frac4v}$.  Interestingly, such frames can never be real: with the exception of the trivial $1\times 1$ and $2\times 2$ cases, the dimensions of all real Hadamard matrices are divisible by $4$; since $v\equiv 1,4\mod 12$, the requisite matrices $H$ here are of size $\smash{\frac{v+2}{3}}\equiv1,2\mod4$.

\subsubsection{Five element blocks: $(2,5,v)$-Steiner ETFs for any $v\equiv 1,5\mod 20$.}

It is known that $(2,5,v)$-Steiner systems exist precisely when $v\equiv 1,5\mod 20$~\cite{AbelG07}.  The corresponding ETFs consist of $\smash{\frac{v(v+3)}{4}}$ vectors in $\smash{\frac{v(v-1)}{20}}$-dimensional space, having redundancy $\smash{5\frac{v+3}{v-1}}$ and density $\smash{\frac5v}$.  Such frames can be real whenever there exists a real Hadamard matrix of size $\frac{v+3}4$.  In particular, letting $v=45$, we see that there exists a real Steiner ETF of $540$ vectors in $99$-dimensional space, a fact not obtained from any other known infinite family.

\subsubsection{Affine geometries: $(2,q,q^n)$-Steiner ETFs for any prime power $q$, $n\geq 2$.}

At this point, the constructions depart from those previously considered, allowing both $k$ and $v$ to vary.  In particular, using techniques from finite geometry, one can show that for any prime power $q$ and any $n\geq2$, there exists a $(2,k,v)$-Steiner system with $k=q$ and $v=q^n$~\cite{ColbournM07}.  The corresponding ETFs consist of $\smash{q^n(1+\frac{q^n-1}{q-1})}$ vectors in $\smash{q^{n-1}(\frac{q^n-1}{q-1}})$-dimensional space.  Like the preceding four classes of Steiner ETFs, these frames can grow arbitrarily large: fixing any prime power $q$, one may manipulate $n$ to produce ETFs of varying orders of magnitude.  However, unlike the four preceding classes, these affine Steiner ETFs also provide great flexibility in choosing one's redundancy.  That is, they provide the ability to pick $M$ and $N$ somewhat independently.  Indeed, the redundancy of such frames $q(1+\frac{q-1}{q^n-1})$ is essentially $q$, which may be an arbitrary prime power.  Moreover, as these frames grow large, they also become increasingly sparse: their density is $\smash{\frac1{q^{n-1}}}$.  Because of their high sparsity and flexibility with regards to size and redundancy, these frames, along with their projective geometry-based cousins detailed below, are perhaps the best known candidates for use in ETF-based applications.  Such ETFs can be real if there exists a real Hadamard matrix of size $1+\frac{q^n-1}{q-1}$, such as whenever $q=2$, or when $q=5$ and $n=3$.

\subsubsection{Projective geometries: $(2,q+1,\frac{q^{n+1}-1}{q-1})$-Steiner ETFs for any prime power $q$, $n\geq 2$.}

With finite geometry, one can show that for any prime power $q$ and any $n\geq2$, there exists a $(2,k,v)$-Steiner system with $k=q+1$ and $\smash{v=\frac{q^{n+1}-1}{q-1}}$~\cite{ColbournM07}.  Qualitatively speaking, the ETFs these projective geometries generate share much in common with their affinely-generated cousins, possessing very high sparsity and great flexibility with respect to size and redundancy.  The technical details are as follows: they consist of $\smash{\frac{q^{n+1}-1}{q-1}(1+\frac{q^n-1}{q-1})}$ vectors in $\smash{\frac{(q^n-1)(q^{n+1}-1)}{(q+1)(q-1)^2}}$-dimensional space, with density $\smash{\frac{q^2-1}{q^{n+1}-1}}$ and redundancy $\smash{(q+1)(1+\frac{q-1}{q^n-1})}$.  These frames can be real if there exists a real Hadamard matrix of size $1+\frac{q^n-1}{q-1}$; note this restriction is identical to that for ETFs generated by affine geometries for the same $q$ and $n$, implying that real Steiner ETFs generated by finite geometries always come in pairs, such as the $6\times 16$ and $7\times 28$ ETFs generated when $q=2$, $n=2$, and the $28\times 64$ and $35\times 120$ ETFs generated when $q=2$, $n=3$.

\subsubsection{Unitals: $(2,q+1,q^3+1)$-Steiner ETFs for any prime power $q$.}

For any prime power $q$, one can show that there exists a $(2,k,v)$-Steiner system with $k=q+1$ and $v=q^3+1$~\cite{ColbournM07}.  Though one may pick a redundancy of one's liking, such a choice confines one to ETFs of a given size: they consist of $(q^2+1)(q^3+1)$ vectors in $\smash{\frac{q^2(q^3+1)}{q+1}}$-dimensional space, having redundancy $\smash{(q+1)(1+\frac1{q^2})}$ and density $\smash{\frac{q+1}{q^3+1}}$.  These ETFs can never be real: the requisite Hadamard matrices are of size $q^2+1$ which is never divisible by $4$ since $0$ and $1$ are the only squares in $\bbZ_4$.

\subsubsection{Denniston designs: $(2,2^r,2^{r+s}+2^r-2^s)$-Steiner ETFs for any $2\leq r<s$.}

For any $2\leq r<s$, one can show that there exists a $(2,k,v)$-Steiner system with $k=2^r$ and $v=2^{r+s}+2^r-2^s$~\cite{ColbournM07}.  By manipulating $r$ and $s$, one can independently determine the order of magnitude of one's redundancy and size, respectively: the corresponding ETFs consist of $(2^s+2)(2^{r+s}+2^r-2^s)$ vectors in $\smash{\frac{(2^s+1)(2^{r+s}+2^r-2^s)}{2^r}}$-dimensional space, having redundancy $\smash{2^r\frac{2^s+2}{2^s+1}}$ and density $\smash{\frac{2^r}{2^{r+s}+2^r-2^s}}$.  As such, this family has some qualitative similarities to the familes of ETFs produced by affine and projective geometries.  However, unlike those families, the ETFs produced by Denniston designs can never be real: the requisite Hadamard matrices are of size $2^s+2$, which is never divisible by $4$.

\begin{table}
\begin{center}
\begin{scriptsize}
\begin{tabular}{llllll}
\hline
Name		&$M$										&$N$											&Redundancy						&Real?					&Restrictions\\
\hline
$2$-blocks	&$\frac{v(v-1)}{2}$							&$v^2$											&$2\frac{v}{v-1}$				&$v$					&None\medskip\\
$3$-blocks	&$\frac{v(v-1)}6$							&$\frac{v(v+1)}{2}$								&$3\frac{v+1}{v-1}$				&$\frac{v+1}{2}$		&$v\equiv 1,3\mod 6$\medskip\\
$4$-blocks	&$\frac{v(v-1)}{12}$						&$\frac{v(v+2)}{3}$								&$4\frac{v+2}{v-1}$				&Never					&$v\equiv 1,4\mod 12$\medskip\\
$5$-blocks	&$\frac{v(v-1)}{20}$						&$\frac{v(v+3)}{4}$								&$5\frac{v+3}{v-1}$				&$\frac{v+3}4$			&$v\equiv 1,5\mod 20$\medskip\\
Affine		&$q^{n-1}(\frac{q^n-1}{q-1})$				&$q^n(1+\frac{q^n-1}{q-1})$						&$q(1+\frac{q-1}{q^n-1})$		&$1+\frac{q^n-1}{q-1}$	&$q$ a prime power, $n\geq2$\medskip\\
Projective	&$\frac{(q^n-1)(q^{n+1}-1)}{(q+1)(q-1)^2}$	&$\frac{q^{n+1}-1}{q-1}(1+\frac{q^n-1}{q-1})$	&$(q+1)(1+\frac{q-1}{q^n-1})$	&$1+\frac{q^n-1}{q-1}$	&$q$ a prime power, $n\geq2$\medskip\\
Unitals		&$\frac{q^2(q^3+1)}{q+1}$					&$(q^2+1)(q^3+1)$								&$(q+1)(1+\frac1{q^2})$			&Never					&$q$ a prime power\medskip\\
Denniston	&$\frac{(2^s+1)(2^{r+s}+2^r-2^s)}{2^r}$		&$(2^s+2)(2^{r+s}+2^r-2^s)$						&$2^r\frac{2^s+2}{2^s+1}$		&Never					&$2\leq r<s$\\
\hline
\end{tabular}
\end{scriptsize}
\caption{\label{table.infinite families}Eight infinite families of Steiner ETFs, each arising from a corresponding known infinite family of $(2,k,v)$-Steiner designs.  Each family permits both $M$ and $N$ to grow very large, but only a few families---affine, projective and Denniston---give one the freedom to simultaneously control the proportion between $M$ and $N$, namely the redundancy $\frac NM$ of the ETF.  The column denoted ``Real?" indicates the size for which a real Hadamard matrix must exist in order for the resulting ETF to be real; it suffices to have this size be a power of $2$; if the Hadamard conjecture is true, it would suffice for this number to be divisible by $4$.}
\end{center}
\end{table}

\begin{table}
\begin{center}
\begin{tabular}{rrrrrcl}
\hline
$M$&$N$&$k$&$v$&$r$&$\bbR/\bbC$&Construction of the Steiner system\\
\hline
  6&  16& 2& 4& 3&$\bbR$&$2$-blocks of $v=4$; Affine with $q=2$, $n=2$\\
  7&  28& 3& 7& 3&$\bbR$&$3$-blocks of $v=7$; Projective with $q=2$, $n=2$\\
 28&  64& 2& 8& 7&$\bbR$&$2$-blocks of $v=8$; Affine with $q=2$, $n=3$\\
 35& 120& 3&15& 7&$\bbR$&$3$-blocks of $v=15$; Projective with $q=2$, $n=3$\\
 66& 144& 2&12&11&$\bbR$&$2$-blocks of $v=12$\\
 99& 540& 5&45&11&$\bbR$&$5$-blocks of $v=45$\\ 
\hline
  3&   9& 2& 3& 2&$\bbC$&$2$-blocks of $v=3$\\
 10&  25& 2& 5& 4&$\bbC$&$2$-blocks of $v=5$\\
 12&  45& 3& 9& 4&$\bbC$&$3$-blocks of $v=9$; Affine with $q=3$, $n=2$\\
 13&  65& 4&13& 4&$\bbC$&$4$-blocks of $v=13$; Projective with $q=3$, $n=2$\\
 15&  36& 2& 6& 5&$\bbC$&$2$-blocks of $v=6$\\
 20&  96& 4&16& 5&$\bbC$&$4$-blocks of $v=16$; Affine with $q=4$, $n=2$\\
 21&  49& 2& 7& 6&$\bbC$&$2$-blocks of $v=7$\\
 21& 126& 5&21& 5&$\bbC$&$5$-blocks of $v=21$; Projective with $q=4$, $n=2$\\
 26&  91& 3&13& 6&$\bbC$&$3$-blocks of $v=13$\\
 30& 175& 5&25& 6&$\bbC$&$5$-blocks of $v=25$; Affine with $q=5$, $n=2$\\
 31& 217& 6&31& 6&$\bbC$&Projective with $q=5$, $n=2$\\
 36&  81& 2& 9& 8&$\bbC$&$2$-blocks of $v=9$\\
 45& 100& 2&10& 9&$\bbC$&$2$-blocks of $v=10$\\
 50& 225& 4&25& 8&$\bbC$&$4$-blocks of $v=25$\\
 55& 121& 2&11&10&$\bbC$&$2$-blocks of $v=11$\\
 56& 441& 7&49& 8&$\bbC$&Affine with $q=7$, $n=2$\\
 57& 190& 3&19& 9&$\bbC$&$3$-blocks of $v=19$\\
 57& 513& 8&57& 8&$\bbC$&Projective with $q=7$, $n=2$\\
 63& 280& 4&28& 9&$\bbC$&Unital with $q=3$; Denniston with $r=2$, $s=3$\\
 70& 231& 3&21&10&$\bbC$&$3$-blocks of $v=21$\\
 72& 640& 8&64& 9&$\bbC$&Affine with $q=8$, $n=2$\\
 73& 730& 9&73& 9&$\bbC$&Projective with $q=8$, $n=2$\\
 78& 169& 2&13&12&$\bbC$&$2$-blocks of $v=13$\\
 82& 451& 5&41&19&$\bbC$&$5$-blocks of $v=41$\\
 90& 891& 9&81&10&$\bbC$&Affine with $q=9$, $n=2$\\
 91& 196& 2&14&13&$\bbC$&$2$-blocks of $v=14$\\
 91&1001&10&91&10&$\bbC$&Projective with $q=9$, $n=2$\\
100& 325& 3&25&12&$\bbC$&$3$-blocks of $v=25$\\ 
\hline
\end{tabular}
\caption{\label{table.low-dimensional examples}The ETFs of dimension 100 or less that can be constructed by applying Theorem~\ref{theorem.steiner etfs} to the eight infinite families of Steiner systems detailed in Section~\ref{section.examples of Steiner ETFs}.  That is, these ETFs represent the first few examples of the general constructions summarized in Table~\ref{table.infinite families}.  For each ETF, we give the dimension $M$ of the underlying space, the number of frame vectors $N$, as well as the number $k$ of elements that lie in any block of a $v$-element set in the corresponding $(2,k,v)$-Steiner system.  We further give the value $r$ of the number of blocks that contain a given point; by Theorem~\ref{theorem.necessary conditions}, $\abs{\ip{f_n}{f_{n'}}}=\frac1r$ measures the angle between any two frame elements.  We also indicate whether the given frame is real or complex, and the method(s) of constructing the corresponding Steiner system.
}
\end{center}
\end{table}

\subsection{Necessary and sufficient conditions on the existence of Steiner ETFs.}

$(2,k,v)$-Steiner systems have been actively studied for over a century, with many celebrated results.  Nevertheless, much about these systems is still unknown.  In this subsection, we discuss some known partial characterizations of the Steiner systems which lie outside of the eight families we have already discussed, as well as what these results tell us about the existence of certain ETFs.  To begin, recall that, for a given $k$ and $v$, if a $(2,k,v)$-Steiner system exists, then the number $r$ of blocks that contain a given point is necessarily $\smash{\frac{v-1}{k-1}}$, while the total number of blocks $b$ is $\smash{\frac{v(v-1)}{k(k-1)}}$.  As such, in order for a $(2,k,v)$-Steiner system to exist, it is necessary for $(k,v)$ to be \textit{admissible}, that is, to have the property that $\smash{\frac{v-1}{k-1}}$ and $\smash{\frac{v(v-1)}{k(k-1)}}$ are integers.

However, this property is not sufficient for existence: it is known that a $(2,6,16)$-Steiner system does not exist~\cite{AbelG07} despite the fact that $\smash{\frac{v-1}{k-1}}=3$ and $\smash{\frac{v(v-1)}{k(k-1)}}=8$.  In fact, letting $v$ be either $16$, $21$, $36$, or $46$ results in an admissible pair with $k=6$, despite the fact that none of the corresponding Steiner systems exist; there are twenty-nine additional values of $v$ which form an admissible pair with $k=6$ and for which the existence of a corresponding Steiner system remains an open problem~\cite{AbelG07}.  Similar nastiness arises with $k\geq 7$.  The good news is that admissibility, though not sufficient for existence, is, in fact, asymptotically sufficient: for any fixed $k$, there exists a corresponding admissible index $v_0(k)$ for which for all $v>v_0(k)$ such that $\smash{\frac{v-1}{k-1}}$ and $\smash{\frac{v(v-1)}{k(k-1)}}$ are integers, a $(2,k,v)$-Steiner system indeed exists~\cite{AbelG07}.  Moreover, explicit values of $v_0(k)$ are known for small $k$: $v_0(6)=801$, $v_0(7)=2605$, $v_0(8)=3753$, $v_0(9)=16497$.  We now detail the ramifications of these design-theoretic results on frame theory:
\begin{thm}
\label{theorem.necessary conditions}
If an $N$-element Steiner equiangular tight frame exists for an $M$-dimensional space, then letting $\smash{\alpha=(\frac{N-M}{M(N-1)})^{\frac12}}$, the corresponding block design has parameters:
\begin{equation*}
v=\tfrac{N\alpha}{1+\alpha},
\qquad
b=M,
\qquad
r=\tfrac1{\alpha},
\qquad
k=\tfrac{N}{M(1+\alpha)}.
\end{equation*}
In particular, if such a frame exists, then these expressions for $v$, $k$ and $r$ are necessarily integers.\medskip

\noindent
Conversely, for any fixed $k\geq2$, there exists an index $v_0(k)$ for which for all $v>v_0(k)$ such that $\smash{\frac{v-1}{k-1}}$ and $\smash{\frac{v(v-1)}{k(k-1)}}$ are integers, there exists a Steiner equiangular tight frame of $\smash{v(1+\frac{v-1}{k-1})}$ vectors for a space of dimension $\smash{\frac{v(v-1)}{k(k-1)}}$.\medskip

\noindent
In particular, for any fixed $k\geq2$, letting $v$ be either $jk(k-1)+1$ or $jk(k-1)+k$ for increasingly large values of $j$ results in a sequence of Steiner equiangular tight frames whose redundancy is asymptotically $k$; these frames can be real if there exist real Hadamard matrices of sizes $jk+1$ or $jk+2$, respectively.
\end{thm}

\begin{proof}
To prove the necessary conditions on $M$ and $N$, recall that Steiner ETFs, namely those ETFs produced by Theorem~\ref{theorem.steiner etfs}, have $\smash{N=v(1+\frac{v-1}{k-1})}$ and $\smash{M=\frac{v(v-1)}{k(k-1)}}$.  Together, these two equations imply $N=v+kM$.  Solving for $k$, and substituting the resulting expression into $\smash{N=v(1+\frac{v-1}{k-1})}$ yields the quadratic equation $0=(M-1)v^2+2(N-M)v-N(N-M)$.  With some algebra, the only positive root of this equation can be found to be $\smash{v=\frac{N\alpha}{1+\alpha}}$, as claimed.  Substituting this expression for $v$ into $N=v+kM$ yields $\smash{k=\tfrac{N}{M(1+\alpha)}}$.  Having $v$ and $k$, the previously-mentioned relations $bk=vr$ and $v-1=r(k-1)$ imply $\smash{r=\frac{v-1}{k-1}=\frac1\alpha}$ and $\smash{b=\frac vkr=M}$, as claimed. 

The second set of conclusions is the result of applying Theorem~\ref{theorem.steiner etfs} to the aforementioned $(2,k,v)$-Steiner ETFs that are guaranteed to exist for all sufficiently large $v$, provided $\smash{\frac{v-1}{k-1}}$ and $\smash{\frac{v(v-1)}{k(k-1)}}$ are integers.  The final set of conclusions are then obtained by applying this fact in the special cases where $v$ is either $jk(k-1)+1$ or $jk(k-1)+k$.  In particular, if $v=jk(k-1)+1$ then $\smash{\frac{v-1}{k-1}=jk}$ and $M=\smash{\frac{v(v-1)}{k(k-1)}=j[jk(k-1)+1]}$ are integers, and the resulting ETF of $(jk+1)[jk(k-1)+1]$ vectors has a redundancy of $\smash{k+\frac1j}$ that tends to $k$ for large $j$; such an ETF can be real if there exists a real Hadamard matrix of size $jk+1$.  Meanwhile, if $v=jk(k-1)+k$ then $\smash{\frac{v-1}{k-1}=jk+1}$ and $M=\smash{\frac{v(v-1)}{k(k-1)}=(jk+1)[j(k-1)+1]}$ are integers, and the resulting ETF of $k(jk+2)[j(k-1)+1]$ vectors has a redundancy of $\smash{k\frac{jk+2}{jk+1}}$ that tends to $k$ for large $j$; such an ETF can be real if there exists a real Hadamard matrix of size $jk+2$.
\end{proof}
We conclude this section with a few thoughts on Theorems~\ref{theorem.steiner etfs} and~\ref{theorem.necessary conditions}.  First, we emphasize that the method of Theorem~\ref{theorem.steiner etfs} is a method for constructing some ETFs, and by no means constructs them all.  Indeed, as noted above, the redundancy of Steiner ETFs is always strictly greater than $2$; while some of those ETFs with $\frac NM<2$ will be the Naimark complements of Steiner ETFs, one must admit that the Steiner method contributes little towards the understanding of those ETFs with $\frac NM=2$, such as those arising from Paley graphs~\cite{Waldron09}.  Moreover, Theorem~\ref{theorem.necessary conditions} implies that not even every ETF with $\frac NM>2$ arises from a Steiner system: though there exists an ETF of $76$-elements in $\bbR^{19}$~\cite{Waldron09}, the corresponding parameters of the design would be $v=\frac{38}3$, $r=5$ and $k=\frac{10}3$, not all of which are integers.

That said, the method of Theorem~\ref{theorem.steiner etfs} is truly significant: comparing Table~\ref{table.low-dimensional examples} with a comprehensive list of all real ETFs of dimension $50$ or less~\cite{Waldron09}, we see the Steiner method produces $4$ of the $17$ ETFs that have redundancy greater than $2$, namely $6\times 16$, $7\times 28$, $28\times 64$ and $35\times 120$ ETFs.  Interestingly, an additional $4$ of these $17$ ETFs can also be produced by the Steiner method, but only in complex form, namely those of $15\times 36$, $20\times 96$, $21\times 126$ and $45\times 100$ dimensions; it is unknown whether this is the result of a deficit in our analysis or the true non-existence of real-valued Steiner-based constructions of these sizes.  The plot further thickens when one realizes that an additional $2$ of these $17$ real ETFs satisfy the necessary conditions of Theorem~\ref{theorem.necessary conditions}, but that the corresponding $(2,k,v)$-Steiner systems are known to not exist: if a $28\times 288$ ETF was to arise as a result of Theorem~\ref{theorem.steiner etfs}, the corresponding Steiner system would have $k=6$ and $v=36$, while the $43\times 344$ ETF would have $k=7$ and $v=43$; in fact, $(2,6,36)$- and  $(2,7,43)$-Steiner systems cannot exist~\cite{AbelG07}.  With our limited knowledge of the rich literature on Steiner systems, we were unable to resolve the existence of two remaining candidates: $23\times 276$ and $46\times 736$ ETFs could potentially arise from $(2,10,46)$- and $(2,14,92)$-Steiner systems, respectively, provided they exist.

\section{Equiangular tight frames and the restricted isometry property}

In the previous section, we used Theorem~\ref{theorem.steiner etfs} to construct many examples of Steiner ETFs.  In this section, we investigate the feasibility of using such frames for compressed sensing applications.  Here, we identify a frame $F=\set{f_n}_{n=1}^{N}$ in $\bbH_M$ with its \textit{synthesis} operator $F:\bbC^N\rightarrow\bbH_M$, $Fg:=\sum_{n=1}^{N}g(n)f_n$.  That is, $F$ is an $M\times N$ matrix whose columns are the $f_n$'s.  For a given $\delta$ and $K$, such an operator $F$ is said to have the \textit{$(K,\delta)$-restricted isometry property} ($(K,\delta)$-RIP) if:
\begin{equation}
\label{equation.definition of RIP}
(1-\delta)\|g\|_2^2\leq\|Fg\|_2^2\leq(1+\delta)\|g\|_2^2
\end{equation}
for all $g\in\bbC^N$ that are \textit{$K$-sparse}, that is, for which $g(n)\neq0$ for at most $K$ values of $n$.  The central problem of compressed sensing is to efficiently solve the underdetermined linear system $Fg=f$ for $g$, given that $f$ itself arises as $f=Fg_0$ where $g_0$ is $K$-sparse.  Here, the true challenge is that, despite the fact that $f$ is a linear combination of at most $K$ of the $f_n$'s, one does not know \textit{a priori} which particular $K$ vectors were employed.  Moreover, since the values for $N$ and $K$ encountered in applications are typically very large, it is not computationally feasible to check every $K$-subset of $\set{f_n}_{n=1}^N$.  It is therefore a remarkable fact~\cite{CandesT05} that $g=g_0$ can indeed be efficiently recovered from the system $Fg=f$ using linear programming, provided the operator $F$ is $(2K,\delta)$-RIP for some $\delta<\sqrt{2}-1$.

Since RIP is such an exceedingly nice property, it is natural to ask whether such matrices even exist---they do.  In fact, \cite{CandesT06} used concentration-of-measure arguments to show that for every $\delta>0$, there exists a constant $C$ such that $M\times N$ matrices of Gaussian or Bernoulli ($\pm1$) entries are $(K,\delta)$-RIP with high probability provided $M\geq CK\log(N/M)$.  Similarly, matrices formed by taking random rows of a Fourier matrix satisfy RIP with high probability when $M\geq CK\log^4(N)$~\cite{RudelsonV08}.  These existence results have spurred a great deal of interest in deterministic RIP matrix constructions that have $M=\rmO(K\log^\beta N)$ for some $\beta\geq1$, but no such constructions are known to date.  Instead, the best known deterministic constructions, such as the one given in~\cite{DeVore07}, have $M=\rmO(K^2)$.  Despite the fact that all ETFs do indeed match this state-of-the-art level of performance, sadly some ETFs---Steiner ETFs in particular---fail to do any better.

To clarify, let $\calK$ be any $K$-element subsequence of $\set{1,\dotsc,N}$, and let $F_\calK:=\set{f_n}_{n\in\calK}$ be the corresponding $M\times K$ submatrix of $F$.  Using a standard argument, one may show that the $(K,\delta)$-RIP condition~\eqref{equation.definition of RIP} is equivalent to having the spectrum of each sub-Gramian $F_\calK^*F_\calK^{}$ lie in the interval $[1-\delta,1+\delta]$; in frame parlance, this implies that each $F_\calK$ is a good \textit{Riesz basis}.  Letting $\rho(A)$ denote the spectral radius of a given $K\times K$ matrix $A$, the $(K,\delta)$-RIP condition is equivalent to having $\rho(F_\calK^*F_\calK^{}-\rmI)\leq\delta$ for all $\calK$.  The problem of constructing RIP matrices thus reduces to one of spectral estimation.  At this point, most research on constructing RIP matrices falls back on a simple, but effective tool: \textit{Gershgorin circles}, namely the fact that every eigenvalue of $A$ lies in one of the disks in the complex plane centered at $a_{k,k}$ and having radius $\sum_{k'\neq k}\abs{a_{k,k'}}$.  When the $f_n$'s have unit norm, as in the case of ETFs, the diagonal entries of the sub-Gramian $F_\calK^*F_\calK^{}$ are all one, and so the application of Gershgorin's estimate to $A=F_\calK^*F_\calK^{}-\rmI$ reduces to the so-called \textit{worst-case-coherence} bound:
\begin{equation*}
\max_{\calK}\rho(F_\calK^*F_\calK^{}-\rmI)
=\max_{\calK}\max_{n\in\calK}\sum_{\substack{n'\in K\\n'\neq n}}\abs{\ip{f_n}{f_{n'}}}
\leq(K-1)\max_{n\neq n'}\abs{\ip{f_n}{f_{n'}}}.
\end{equation*}
In particular, in order for $F=\set{f_n}_{n=1}^N$ to be $(K,\delta)$-RIP, it suffices to have:
\begin{equation}
\label{equation.RIP sufficient coherence condition}
(K-1)\max_{n\neq n'}\abs{\ip{f_n}{f_{n'}}}\leq\delta.
\end{equation}
Further note that when using \eqref{equation.RIP sufficient coherence condition} to demonstrate RIP for any fixed $\delta<1$, the largest possible values of $K$ occur when $\abs{\ip{f_n}{f_{n'}}}$ achieves its lower \textit{Welch bound} $\smash{(\frac{N-M}{M(N-1)})^{\frac12}}$, which occurs precisely when $F$ is an ETF~\cite{StrohmerH03}.  We summarize the preceding discussion as follows:
\begin{thm}
\label{theorem.RIPETF coherence}
For any fixed $\delta<1$, an equiangular tight frame $F=\set{f_n}_{n=1}^{N}$ in $\bbH_M$ has the $(K,\delta)$-restricted isometry property~\eqref{equation.definition of RIP} for all $K\leq 1+\delta(\tfrac{M(N-1)}{N-M})^{\frac12}$.\medskip
 
\noindent
Moreover, for any unit norm $f_n$'s, no argument that relies on the worst-case-coherence-based bound~\eqref{equation.RIP sufficient coherence condition} can provide a better range for such $K$.
\end{thm}
Note that when $N\geq 2M$, we have $\smash{1\leq\frac{N-1}{N-M}\leq2}$ and therefore the maximum permissible value of $K$ is on the order of $M^\frac{1}{2}$, which is consistent with other known deterministic constructions of RIP matrices.  This is not to say that ETFs, in general, cannot be RIP with $M=\rmO(K\log^\beta N)$ for some $\beta\geq1$, but rather, that such a fact cannot be obtained using the worst-case-coherence-based bound~\eqref{equation.RIP sufficient coherence condition}.  

This hope notwithstanding, one of the sad consequences of the Steiner construction method of Theorem~\ref{theorem.steiner etfs} is that we, for the first time, know that there is a large class of ETFs for which the seemingly coarse estimate~\eqref{equation.RIP sufficient coherence condition} is, in fact, accurate.  In particular, recall from Theorem~\ref{theorem.steiner etfs} that every Steiner ETF is built from carefully overlapping $v$ regular simplices, each consisting of $r+1$ vectors in a $r$-dimensional subspace of $b$-dimensional space.  In particular, letting $K=r+1$, the corresponding subcollection of all $K$ vectors that lie in a given block are linearly dependent, which, in accordance with~\eqref{equation.definition of RIP}, forces the corresponding $\delta$ to be at least $1$.  Recalling the value of $r$ given in Theorem~\ref{theorem.necessary conditions}, we see that Steiner ETFs cannot be $(K,\delta)$-RIP for any $\delta<1$ so long as $K$ is at least $1+(\tfrac{M(N-1)}{N-M})^{\frac12}$.  That is, for Steiner ETFs, the best one can truly do is, in fact, given by Theorem~\ref{theorem.RIPETF coherence}.  This begs the open question: are there any ETFs which are RIP with $M=\rmO(K\log^\beta N)$, or does optimizing a coarse bound---worst-case-coherence---always come at the cost of being able to realize a truly small spectral radius?

\section*{Acknowledgments}
The authors thank Prof.~Peter G.~Casazza for his insightful suggestions.  Fickus was supported by NSF DMS 1042701, NSF CCF 1017278,  AFOSR F1ATA00083G004 and AFOSR F1ATA00183G003.  Mixon was supported by the A.B. Krongard Fellowship.  Tremain was supported by NSF DMS 1042701.  The views expressed in this article are those of the authors and do not reflect the official policy or position of the United States Air Force, Department of Defense, or the U.S. Government.

\end{document}